\documentclass{amsart}
	\usepackage[latin1]{inputenc}
	\usepackage{ae,aecompl,amsbsy,amssymb,amsmath,amsthm,
eurosym,amsfonts,epsfig,graphicx,graphics,verbatim,enumerate,esint,color,MnSymbol}

\usepackage{color,soul}

\definecolor{lightblue}{rgb}{.85,.93,1}
\sethlcolor{lightblue}

\newcommand{\be}{\begin{equation*}}
\newcommand{\ee}{\end{equation*}}
\newcommand{\ba}{\begin{eqnarray*}}
\newcommand{\ea}{\end{eqnarray*}}

\newcommand{\rn}{\mathbb{R}^d}

\newcommand{\Q}{|Q|}

\newcommand{\abs}[1]{\lvert#1\rvert}

\theoremstyle{plain}
\newtheorem{thm}{Theorem}[section]
\newtheorem{lemma}[thm]{Lemma}

\theoremstyle{definition}

\newtheorem*{notation}{Notation}

\theoremstyle{definition}

\theoremstyle{remark}
\newtheorem{remark}{Remark}[section]

\newtheorem*{claim}{Claim}

\numberwithin{equation}{section}

\newcommand{\normllrp}[1]{\lvert#1\rvert_{r\prime}}
\newcommand{\norm}[1]{\lVert#1\rVert}

\newcommand{\chf}[1]{\textup{ch}_{\mathcal{F}}(#1)}
\newcommand{\chsf}[1]{\textup{ch}^*_{\mathcal{F}}(#1)}
\newcommand{\chg}[1]{\textup{ch}_{\mathcal{G}}(#1)}
\newcommand{\chsg}[1]{\textup{ch}^*_{\mathcal{G}}(#1)}
\newcommand{\eg}[1]{E_{\mathcal{G}}(#1)}
\newcommand{\ef}[1]{E_{\mathcal{F}}(#1)}

\newcommand{\pif}[1]{\pi_\mathcal{F}(#1)}

\newcommand{\pig}[1]{\pi_\mathcal{G}(#1)}
\newcommand{\av}[1]{\langle #1\rangle}
\newcommand{\avfs}{\langle f\rangle^\sigma}

\newcommand{\avog}{\langle \normllrp{g}\rangle^\omega}

\newcommand{\avo}[1]{\langle #1\rangle^\omega}

\newcommand{\avs}[1]{\langle #1\rangle^\sigma}

\newcommand{\avgg}{\langle \normllrp{g}\rangle^\omega_G\;}

\newcommand{\avon}[1]{\langle \normllrp{#1}\rangle^\omega}

\newcommand{\dc}{\mathcal{D}}
\newcommand{\pairo}[2]{\langle #1,#2 \rangle_{L^p_{\ell^r}(\omega)\times L^{p'}_{\ell^{r'}}(\omega)}}
\newcommand{\pairorealvalued}[2]{\langle #1,#2 \rangle_{L^p(\omega)\times L^{p'}(\omega)}}

\newcommand{\pairs}[2]{\langle #1,#2 \rangle_{L^{p}(\sigma)\times L^{p'}(\sigma)}}
\newcommand{\normlps}[1]{\lVert#1\rVert_{L^p(\sigma)}}

\newcommand{\normlppo}[1]{\lVert#1\rVert_{L^{p'}_{\ell^{r'}}(\omega)}}

\newcommand{\cf}{\mathcal{F}}
\newcommand{\ci}{\mathcal{I}}

\newcommand{\cg}{\mathcal{G}}
\newcommand{\cd}{\mathcal{D}}

\newcommand{\normrp}[1]{\lvert#1\rvert_{r'}}
\newcommand{\normr}[1]{\lvert#1\rvert_{r}}

\begin{document}

\date{\today}
     

\title[Another proof of Scurry's two weight norm inequality]{Another proof of Scurry's characterization of a two weight norm inequality for a sequence-valued positive dyadic operator}

\author{Timo S. H\"anninen}
\address{Department of Mathematics and Statistics, University of Helsinki, P.O. Box 68, FI-00014 HELSINKI, FINLAND}
\thanks{The author is supported by the European Union through the ERC Starting Grant "Analytic-probabilistic methods for borderline singular integrals". }
\email{timo.s.hanninen@helsinki.fi}

\begin{abstract}
In this note we prove Scurry's testing conditions for the boundedness of a sequence-valued averaging positive dyadic operator from a weighted Lp space to a sequence-valued weighted Lp space by using parallel stopping cubes.
\end{abstract}
\maketitle
\tableofcontents
\section{Introduction and statement of the theorem}
Let $\lambda_Q$ be non-negative real numbers indexed by the dyadic cubes $Q\in\dc$ of $\rn$. We define the operator $T$ by
$$
T(f):=(\lambda_Q\av{f}_Q1_Q)_{Q\in\dc}.
$$
Suppose that $1<p<\infty$. Let $u$ and $\omega$ be weights. 
We are considering sufficient and necessary testing conditions for the boundedness of the operator $T:L^p(u)\to L^p_{\ell^r}(\omega)$. By the change of weight $\sigma=u^{-1/(p-1)}$ we may as well study the boundedness of the operator $T(\,\cdot\,\sigma):L^p(\sigma)\to L^p_{\ell^r}(\omega)$. 

In the case $r=\infty$ Sawyer \cite[Theorem A]{sawyer1982} proved that for $\lambda_Q=\abs{Q}^{a/d}$ with $0\leq a <d$ it is sufficient to test the boundedness of the operator $T(\,\cdot\,\sigma):L^p(\sigma)\to L^p_{\ell^\infty}(\omega)$ on functions $f=1_R$ with $R\in\cd$. This testing condition holds for every $\lambda_Q$, as one can check by using the well-known proof in which one linearizes the operator $\abs{Tf}_\infty=\sum_{Q\in\cd}\lambda_Q\av{f}_Q1_{E(Q)}$ by using the partition $$E(Q):=\{x\in Q: \abs{Tf(x)}_\infty=\lambda_Q\av{f}_Q \text{ and } \lambda_{Q'}\av{f}_{Q'}< \lambda_Q\av{f}_Q \text{ whenever } Q'\supsetneq Q\}$$
and applies the dyadic Carleson embedding theorem. The exact statement of the testing condition in the case $r=\infty$ corresponds to Theorem \ref{scurry} with $r=\infty$ and with the dual testing \eqref{dualtesting} omitted. 

In the case $r=1$ the boundedness of the sequence-valued operator $T(\,\cdot\,\sigma):L^p(\sigma)\to L^p_{\ell^1}(\omega)$ is equivalent to the boundedness of the real-valued operator $S(\,\cdot\,\sigma):L^p(\sigma)\to L^p(\omega)$ defined by $$Sf:=\abs{Tf}_1=\sum_{Q\in\cd} \lambda_Q\av{f}_Q1_Q.$$ For the boundedness of $S(\,\cdot\,\sigma):L^p(\sigma)\to L^p(\omega)$ it is sufficient to test the boundedness of both the operator $S(\,\cdot\,\sigma):L^p(\sigma)\to L^p(\omega)$ and its formal adjoint $S(\,\cdot\,\omega):L^{p'}(\omega)\to L^{p'}(\sigma)$ on functions $f=1_R$ with $R\in\cd$. These testing conditions were proven for $p=2$ 
\begin{itemize}
\item by Nazarov, Treil, and Volberg  \cite{nazarov1999} by the Bellman function technique 
\end{itemize}
and for $1<p<\infty$ 
\begin{itemize}
\item by Lacey, Sawyer, and Uriarte-Tuero \cite{lacey2009} by techniques that are similar to the ones that Sawyer \cite{sawyer1988} used in proving such testing conditions for a large class of integral operators $I(\,\cdot\,\sigma):L^p(\sigma)\to L^p(\omega)$ with non-negative kernels (in particular for fractional integrals and Poisson integrals),
\item by Treil \cite{treil2012} by splitting the summation over dyadic cubes $Q\in\cd$ in the dual pairing $\pairorealvalued{Sf}{g}$ by the condition "$\sigma(Q)(\avs{f}_Q)^p>\omega(Q)(\avo{g}_Q)^{p'}"$, 
\item and by  
Hyt\"onen \cite[Section 6]{hytonen2012} by constructing stopping cubes for each of the pairs $(f,\sigma)$ and $(g,\omega)$ in parallel and then splitting the summation in the dual pairing $\pairorealvalued{Sf}{g}$ by the condition "$\pif{Q}\subseteq\pig{Q}$". The technique of organizing the summation by parallel stopping cubes is from the work of Lacey, Sawyer, Shen and Uriarte-Tuero \cite{lacey2012} on the two-weight boundedness of the Hilbert transform. 
\end{itemize}
The exact statement of the testing conditions for the operator $S(\,\cdot\,\sigma):L^p(\sigma)\to L^p(\omega)$ corresponds to Theorem \ref{caserone}, which is equivalent to Theorem \ref{scurry} with $r=1$ for the operator $T(\,\cdot\,\sigma):L^p(\sigma)\to L^p_{\ell^1}(\omega)$, as explained in Remark \ref{operators}.

In the case $1<r<\infty$ the testing conditions in Theorem \ref{scurry} for the boundedness of the operator $T(\,\cdot\,\sigma):L^p(\sigma)\to L^p_{\ell^r}(\omega)$ were first proven by Scurry \cite{scurry2010} by adapting Lacey, Sawyer, and Uriarte-Tuero's proof of the case $r=1$ to the case $1<r<\infty$. In this note we adapt Hyt\"onen's proof of the case $r=1$ to the case $1<r<\infty$.

Next we state Theorem \ref{scurry}. Note that the formal adjoint $T^*:L^{p'}_{\ell^{r'}}\to L^{p'}$ of the operator $T:L^p\to L^p_{\ell^r}$ is given by
$$
T^*(g)=\sum_{Q\in\dc}\lambda_Q\av{g_Q}_Q1_Q.
$$
The operator $T$ is positive in the sense that if $f\geq 0$, then $(Tf)_Q\geq0$ for every $Q\in\cd$. Likewise, the operator $T^*$ is positive in the sense that if $g_Q\geq 0$ for every $Q\in\cd$, then $T^*(g)\geq0$.
For each dyadic cube $R$ we define the localized version $T_R$ of the operator $T$ by
$$
T_R(f):=(\lambda_Q\av{f}_Q1_Q)_{\substack{Q\in\dc:\\Q\subseteq R}}.
$$
Hence the formal adjoint $T_R^*:L^{p'}_{\ell^{r'}}\to L^{p'}$ of the operator $T_R:L^p\to L^p_{\ell^r}$ is given by
$$
T^*_R(g)=\sum_{\substack{Q\in\dc:\\Q\subseteq R}}\lambda_Q\av{g_Q}_Q1_Q.
$$
Note that for the formal adjoint $(T(\,\cdot\,\sigma))^*:L^{p'}_{\ell^{r'}}(\omega)\to L^{p'}(\sigma)$ of the operator $T(\,\cdot\,\sigma):L^p(\sigma)\to L^p_{\ell^r}(\omega)$ we have $(T(\,\cdot\,\sigma))^*=T^*(\,\cdot\,\omega)$.

\begin{thm}\label{scurry}
Let $1\leq r\leq \infty$ and $1<p<\infty$. Suppose that $\sigma$ and $\omega$ are locally integrable positive functions. 
Then the two weight norm inequality
\begin{equation}\label{norminequality}
\norm{T(f\sigma)}_{L^p_{\ell^r}(\omega)}\leq \tilde{C} \norm{f}_{L^p(\sigma)}
\end{equation}
holds if and only if both of the following testing conditions hold
\begin{subequations}
\begin{align}
\norm{T_R(\sigma)}_{L^p_{\ell^r}(\omega)}&\leq C \sigma(R)^{1/p} \quad \text{ for all } R\in\dc\label{testing}\\
\norm{T^*_R(g\,\omega)}_{L^{p'}(\sigma)}&\leq C^* \norm{g}_{L^\infty_{\ell^{r'}}(\omega)} \omega(R)^{1/p'}\quad\text{ for all } g=(a_Q1_Q)_{Q\in\cd} \label{dualtesting}\\
&\phantom{\leq C^* \norm{g}_{L^\infty_{\ell^{r'}}(\omega)} \omega(R)^{1/p'}\quad\text{ for all }g=}\text{ with constants $a_Q\geq0$.} \nonumber
\end{align}
\end{subequations}
Moreover, $\tilde{C}\leq C_{p',r'} 20p\,p'(C+C^*)$, where $C_{p',r'}$ is the constant of Stein's inequality. 
\end{thm}

\begin{remark}[Restrictions on the test functions in the dual testing condition]The dual testing condition \eqref{dualtesting} for all functions $g$ is equivalent to the dual testing condition restricted to functions $g$ such that $\normrp{g(x)}=1$ for $\omega$-almost every $x\in\rn$, which is seen as follows. Suppose that $\norm{g}_{L^\infty_{\ell^{r'}}(\omega)}<\infty$. Then  $\abs{g_Q}\leq \norm{g}_{L^\infty_{\ell^{r'}}(\omega)}\frac{1}{\normrp{g}}\abs{g_Q}$ for every $Q\in\cd$ $\omega$-almost everywhere. Note that $\normrp{\frac{1}{\normrp{g}}(\abs{g_Q})_{Q\in\cd}}=1$ $\omega$-almost everywhere. By the positivity and the linearity of the operator $T^*$ we have
$$
\abs{T^*((g_Q)_{Q\in\cd}\omega)}\leq T^*((\abs{g_Q})_{Q\in\cd}\omega)\leq \norm{g}_{L^\infty_{\ell^{r'}}(\omega)} T^*(\frac{1}{\normrp{g}}(\abs{g_Q})_{Q\in\cd}\omega).
$$
Moreover, the dual testing condition \eqref{dualtesting} for all functions $g=(g_Q)_{Q\in\cd}$ is equivalent to the dual testing condition restricted to piecewise constant functions $g=(a_Q1_Q)_{Q\in\cd}$, as observed in Section \ref{reductionobservation}.
\end{remark}
\begin{remark}[Sufficient condition for the dual condition]\label{remarkonsufficientcondition}
The condition 
\begin{equation}\label{extracond}
\norm{T_R(\omega)}_{L^{p'}_{\ell^r}(\sigma)}\leq C^* \omega(R)^{1/{p'}} \quad \text{ for all } R\in\dc
\end{equation}
implies the dual testing condition \eqref{dualtesting}. This is seen as follows. We have that
\begin{align*}
&T^*_R((a_Q1_Q)_{Q\in\cd}\,\omega)&&\\
&=\sum_{\substack{Q\in\cd:\\Q\subseteq R}}\lambda_Qa_Q1_Q\av{\omega}_Q&&\text{by the definition of $T^*_R$}\\
&\leq \abs{(a_Q1_Q)_{Q\in\cd}}_{r'}\normr{(\lambda_Q\av{\omega}_Q1_Q)_{{\begin{subarray}{l}Q\in\cd:\\Q\subseteq R\end{subarray}}}}&&\text{by H\"older's inequality}\\
&\leq \norm{(a_Q1_Q)}_{L^\infty_{\ell^{r'}}} \normr{T_R(\omega)}&&\text{by the definition of $T_R$.}
\end{align*}
Hence by \eqref{extracond} we have
\begin{align*}
\norm{T^*_R((a_Q1_Q)_{Q\in\cd}\,\omega)}_{L^{p'}(\sigma)}&\leq \norm{(a_Q1_Q)}_{L^\infty_{\ell^{r'}}} \norm{T_R(\omega)}_{L^{p'}_{\ell^r}(\sigma)}&& \\
&\leq C^*\norm{(a_Q1_Q)}_{L^\infty_{\ell^{r'}}(\omega)}\omega(R)^{1/{p'}}.&&
\end{align*}

\end{remark}
\begin{remark}[In the case $r=1$ we may consider a real-valued operator]\label{operators}

Consider the real-valued operator $S$ defined by
$$
Sf:=\abs{Tf}_1=\sum_{Q\in\cd} \lambda_ Q \av{f}_Q1_Q.
$$
Note that in this notation the direct testing condition \eqref{testing} is written as
$$
\norm{S_R(\sigma)}_{L^p(\omega)}\leq C \sigma(R)^{1/p}.
$$ 
Observe that the operator $S:L^p\to L^p$ is formally self-adjoint and that for the adjoint $(S(\,\cdot\,\sigma))^*:L^{p'}(\omega)\to L^{p'}(\sigma)$ of the operator $S(\,\cdot\,\sigma):L^p(\sigma)\to L^p(\omega)$ we have $S(\,\cdot\,\sigma))^*=S(\,\cdot\,\omega)$.
By Remark \ref{remarkonsufficientcondition} the dual testing condition \eqref{dualtesting} is implied by the dual testing condition 
\begin{equation}\label{dualtestingconditionfortheoperators}
\norm{S_R(\omega)}_{L^{p'}(\sigma)}\leq C^* \omega(R)^{1/{p'}},
\end{equation}
and, conversely, the dual testing condition \eqref{dualtestingconditionfortheoperators} is implied by the dual testing condition \eqref{dualtesting} applied to the function $g=(1_Q)_{Q\in\cd}$. Therefore Theorem \ref{scurry} in the case $r=1$ is equivalent to the following theorem.

\begin{thm}\label{caserone}
Let $1<p<\infty$. Suppose that $\sigma$ and $\omega$ are locally integrable positive functions. 
Then the two weight norm inequality
\begin{equation}
\norm{S(f\sigma)}_{L^p(\omega)}\leq \tilde{C} \norm{f}_{L^p(\sigma)}
\end{equation}
holds if and only if both of the following testing conditions hold
\begin{subequations}
\begin{align}
\norm{S_R(\sigma)}_{L^p(\omega)}&\leq C \sigma(R)^{1/p} \quad \text{ for all } R\in\dc\\
\norm{S_R(\omega)}_{L^{p'}(\sigma)}&\leq C \omega(R)^{1/{p'}} \quad \text{ for all } R\in\dc.
\end{align}
\end{subequations}
\end{thm}
\end{remark}

\section{Proof of the theorem in the case $1\leq r\leq\infty$}
\begin{notation}
We use the following standard notation: $\avs{f}_Q:=\frac{1}{\sigma(Q)}\int_Q f\sigma$, $\av{f}_Q:=\frac{1}{\abs{Q}}\int_Q f$, and $\normrp{g}:=\norm{g}_{\ell^{r\prime}}$. 
\end{notation}
\begin{proof}[Proof of the necessity of the testing conditions]

By duality the norm inequality \eqref{norminequality} for the operator $T(\,\cdot\, \sigma):L^p(\sigma)\to L^{p}_{\ell^r}(\omega)$ is equivalent to the norm inequality
\begin{equation}\label{dualnorminequality}
\norm{T^*(g\omega)}_{L^{p'}(\sigma)}\leq \tilde{C}  \norm{g}_{L^{p'}_{\ell^{r'}}(\omega)}
\end{equation}
for the adjoint operator $T^*(\,\cdot\,\omega):L^{p'}_{\ell^{r'}}(\omega)\to L^{p'}(\sigma) $. The necessity of the direct testing condition \eqref{testing} follows by applying the norm inequality \eqref{norminequality} for functions $f=1_R$ and the necessity of the dual testing condition \eqref{dualtesting} follows by applying the norm inequality \eqref{dualnorminequality} for functions $g1_R$ and using the estimate $$
\norm{g1_R}_{L^{p'}_{\ell^{r'}}(\omega)}\leq \norm{g}_{L^\infty_{\ell^{r'}}(\omega)} \omega(R)^{1/p'}.
$$
\end{proof}
\begin{proof}[Proof of the sufficiency of the testing conditions] By duality the norm inequality \eqref{norminequality} is equivalent the following norm inequality for the dual pairing
$$
\pairo{T(f\sigma)}{g}\leq \tilde{C} \normlps{f}\normlppo{g}.
$$

\subsection{Reductions}\label{section_reductions}
\begin{claim}[Reduction]We may assume that $f\geq0$, $g_Q\geq0$ for every $Q\in\cd$, and $g=(a_Q1_Q)_{Q\in\cd}$ for some constants $a_Q\geq0$. Moreover, we may assume that the collection $\cd$ is finite and that for some $Q_0\in\cd$ we have $Q\subseteq Q_0$ for all $Q\in\cd$.
\end{claim}
\begin{proof}[Proof of the claim]
Since
$$
\abs{\pairo{T(f\sigma)}{(g_Q)_{Q\in\cd}}}\leq \pairo{T(\abs{f}\sigma)}{(\abs{g_Q})_{Q\in\cd}},
$$
$\normlps{f}=\normlps{\abs{f}}$ and $\normlppo{(g_Q)_{Q\in\cd}}=\normlppo{(\abs{g_Q})_{Q\in\cd}}$, we may assume that $g_Q\geq 0$ and $f\geq0$. By the monotone convergence theorem we may assume that the collection $\cd$ is finite and that all dyadic cubes in the collection $\dc$ are contained in some dyadic cube $Q_0$. We observe that 
\begin{equation}\label{reductionobservation}
\begin{split}
T^*((g_Q)\omega)&=\sum_{Q\in\dc}\lambda_Q\av{g_Q\omega}_Q1_Q=\sum_{Q\in\dc}\lambda_Q\av{g_Q}^\omega_Q\av{\omega}_Q1_Q\\&=\sum_{Q\in\dc}\lambda_Q\av{\av{g_Q}^\omega_Q1_Q\omega}_Q1_Q=T^*((\av{g_Q}^\omega_Q1_Q)\omega)
.\end{split}
\end{equation}
For $1\leq r'\leq \infty$ and $1<p'<\infty$ we have by Stein's inequality 
\begin{equation*}\label{steinsinequality}
\normlppo{(\av{g_Q}^\omega_Q1_Q)}\leq C_{p',r'} \normlppo{(g_Q1_Q)}
\end{equation*}
for \begin{equation}\label{constantinsteininequality}
C_{p',r'}=\begin{cases}{(\frac{p'}{r'}})^{1/{r'}}&\text { if } p'\geq r'\\
(\frac{p}{r})^{1/r}&\text{ if } p'<r'
\end{cases}.
\end{equation}
Hence we may assume that the function $g$ is piecewise constant in the sense that $g=(a_Q1_Q)$ for some constants $a_Q\geq0$.
\end{proof}
\begin{remark}
The constant \eqref{constantinsteininequality} in Stein's inequality can be checked in the following well-known way. Let $(\mathcal{F}_k)_{k\in\mathbb{Z}}$ be a filtration. By Doob's inequality $$\norm{(\mathbb{E}(f\vert\mathcal{F}_k))_{k\in\mathbb{Z}}}_{L^p_{\ell^\infty}}\leq p' \norm{f}_{L^p}$$ for all $1<p\leq \infty$ and for all nonnegative functions $f$. From this it follows directly that $$\norm{(\mathbb{E}(g_k\vert \mathcal{F}_k)_{k\in\mathbb{Z}}}_{L^p_{\ell^\infty}}\leq \norm{(\mathbb{E}(\abs{g_k}_\infty\vert \mathcal{F}_k)_{k\in\mathbb{Z}}}_{L^p_{\ell^\infty}}\leq p' \norm{(g_k)_{k\in\mathbb{Z}}}_{L^p_{\ell^\infty}}$$ and by using duality that 
\begin{equation}\label{vectorvalueddoob}\norm{(\mathbb{E}(g_k\vert \mathcal{F}_k))_{k\in\mathbb{Z}}}_{L^p_{\ell^1}}\leq p \norm{(g_k)_{k\in\mathbb{Z}}}_{L^p_{\ell^1}}
\end{equation} for all $1\leq p<\infty$ and for all nonnegative functions $(g_k)_{k\in\mathbb{Z}}$. 
Hence in the case $p/r\geq 1$ we have by Jensen's inequality and by the inequality \eqref{vectorvalueddoob} that 
\begin{equation*}
\begin{split}
&\norm{(\mathbb{E}(g_k\vert \mathcal{F}_k))_{k\in\mathbb{Z}}}_{L^p_{\ell^r}}=\norm{(\mathbb{E}(g_k\vert \mathcal{F}_k)^r)_{k\in\mathbb{Z}}}^{1/r}_{L^{p/r}_{\ell^1}}\\
&\leq \norm{(\mathbb{E}(g^r_k\vert \mathcal{F}_k))_{k\in\mathbb{Z}}}^{1/r}_{L^{p/r}_{\ell^1}}\leq (\frac{p}{r})^{1/r}\norm{(g^r_k)_{k\in\mathbb{Z}}}^{1/r}_{L^{p/r}_{\ell^1}} =(\frac{p}{r})^{1/r}\norm{(g_k)_{k\in\mathbb{Z}}}_{L^p_{\ell^r}}.
\end{split}
\end{equation*}
Case $p/r< 1$ can be checked by using duality.
\end{remark}
\subsection{Constructing stopping cubes and organizing the summation}
Next we define recursively stopping cubes for each of the pairs $(f,\sigma)$ and $(g,\omega)$.
\begin{claim}[Construction and properties of the stopping cubes related to the pair $(g,\omega)$] Let $\chg{G}$ be the collection of all the maximal dyadic subcubes $G'$ of $G$ such that 
\begin{equation}\label{stoppingg}
\normrp{(a_Q)_{{\begin{subarray}{l}Q\in\cd :\\ Q\supseteq G'\end{subarray}}}}>2\avgg.
\end{equation}
Define recursively $\cg_0:=\{Q_0\}$ and $\cg_{k+1}:=\bigcup_{G\in\cg_k}\chg{G}$. Let $\cg:=\bigcup_{k=0}^\infty \cg_k$. Let 
\begin{equation*}
\label{definition_eg}
\eg{G}:=G\setminus\bigcup_{G'\in\chg{G}}G'.
\end{equation*}
 Define $\pig{Q}$ as the minimal $G\in\cg$ such that $Q\subseteq G$. Then the following properties hold:
\begin{itemize}
\item[(b1)] The sets $\{G'\}_{G'\in\chg{G}}\cup\{\eg{G}\}$ partition $G$.
\item[(b2)] The collection $\{\eg{G}\}_{G\in\cg}$ is pairwise disjoint.
\item[(b3)] $\omega(\eg{G})\geq \frac{1}{2} \omega(G)$.
\item[(b4)] $\normrp{(a_Q)_{\begin{subarray}{l}Q\in\cd:\\ Q\supseteq R\end{subarray}}}\leq 2\avo{\normrp{g}}_{\pig{R}}$
\item[(b5)] $\norm{(g_Q)_{\begin{subarray}{l}Q\in\cd :\\ \pig{Q}=G\end{subarray}}}_{L^{\infty}_{\ell^{r'}}(\omega)}\leq 2 \avgg$.
\end{itemize}
\end{claim}
\begin{proof}[Proof of the claim] The property (b1) holds because $G'\in\chg{G}$ are maximal subcubes of $G$ and $\eg{G}$ is the complement of $\bigcup_{G'\in\chg{G}}G'$ in $G$. Next we check the property (b2). By definition of the set $\eg{G}$ we have that the collection $\{\eg{G}\}\cup\{G'\}_{G'\in\chg{G}}$ is pairwise disjoint. Since $\eg{G'}\subseteq G'$, the collection $\{\eg{G}\}\cup\{\eg{G'}\}_{G'\in\chg{G}}$ is pairwise disjoint. This together with the recursive definition of the collection $\cg$ implies by induction that the collection $\{\eg{G}\}_{G\in\cg}$ is pairwise disjoint.

Next we prove the property (b3). We have
\begin{align*}
\avog_G&=\sum_{G'\in\chg{G}}\avog_{G'}\frac{\omega(G')}{\omega(G)}+\avog_{\eg{G}}\frac{\omega(\eg{G})}{\omega(G)}&&\text{the law of total expectation}\\
&\geq \sum_{G'\in\chg{G}}\avog_{G'}\frac{\omega(G')}{\omega(G)}&&\\
&\geq\sum_{G'\in\chg{G}}\normrp{(a_Q)_{\begin{subarray}{l}Q\in\cd:\\ Q\supseteq G'\end{subarray}}} \frac{\omega(G')}{\omega(G)}&&\\
&\geq 2 \avgg \sum_{G'\in\chg{G}}\ \frac{\omega(G')}{\omega(G)}&&\text{by \eqref{stoppingg}.}
\end{align*}
Hence $$\sum_{G'\in\chg{G}}\omega(G')\leq \frac{1}{2}\omega(G),$$
which by the definition $\eg{G}:=G\setminus\bigcup_{G'\in\chg{G}}G'$ is equivalent to $$\omega(\eg{G})\geq \frac{1}{2} \omega(G).$$ 

Next we prove (b4). Assume that $\pig{R}=G$. By definition this means that $G\in\cg$ is such that $R\subseteq G$ and that there is no $G'\in\cg$ such that $R\subseteq G'\subsetneq G$. If we had
$$\normrp{(a_Q)_{\begin{subarray}{l}Q\in\cd:\\Q\supseteq R\end{subarray}}}>2\avo{\normrp{g}}_G,$$ 
then by definition of the collection $\chg{G}$ there would be $G'\in\chg{G}$ such that $R\subseteq G'$ and $G'\subsetneq G$, in which case $R\subseteq G'\subsetneq G$. Therefore by contrapositive we have

$$\normrp{(a_Q)_{\begin{subarray}{l}Q\in\cd:\\Q\supseteq R\end{subarray}}}\leq2\avo{\normrp{g}}_G.$$ 

Next we prove (b5). Observe that the function $x\mapsto(a_Q1_Q(x))_{\begin{subarray}{l}Q\in\cd :\\ \pig{Q}=G\end{subarray}}$ is supported on $\bigcup_{Q\in\cd:\pig{Q}=G}Q$. Let $x$ be in the support of the function. Let $Q_x$ be the minimal $Q\in\cd$ such that $Q\ni x$ and $\pig{Q}=G$. By the piecewise constancy and the property (b4) we have
$$
\normrp{(a_Q1_Q(x))_{\begin{subarray}{l}Q\in\cd :\\ \pig{Q}=G\end{subarray}}}=\normrp{(a_Q)_{\begin{subarray}{l}Q\in\cd :\\ \pig{Q}=G \text{ and } Q\supseteq Q_x\end{subarray}}}\leq \normrp{(a_Q)_{\begin{subarray}{l}Q\in\cd :\\ Q\supseteq Q_x\end{subarray}}}\leq 2 \avo{\normrp{g}}_{\pig{Q_x}}=2\avo{\normrp{g}}_{G}.$$
This completes the proof of the claim.
\end{proof}

For the pair $(f,\sigma)$ we choose the stopping cubes as in the case $r=1$, which is as follows. Let $\chf{F}$ be the collection of all maximal dyadic subcubes $F'$ of $F$ such that 
\begin{equation}\label{stoppingconditionforf}
\av{f}^\sigma_{F'}>2\av{f}^\sigma_{F}.
\end{equation}
Define recursively $\cf_0:=\{Q_0\}$ and $\cf_{k+1}:=\bigcup_{F\in\cf_k}\chf{F}$. Let $\cf:=\bigcup_{k=0}^\infty \cf_k$.
Let $$\ef{F}:=F\setminus\bigcup_{F'\in\chf{F}}F'.$$ Define $\pif{Q}$ as the minimal $F\in\cf$ such that $Q\subseteq F$.
The construction has the following well-known properties.
\begin{itemize}
\item[(a1)] The sets $F'\in\chf{F}$ and $\ef{F}$ partition $F$.
\item[(a2)] The collection $\{\ef{G}\}_{G\in\cf}$ is pairwise disjoint.
\item[(a3)] $\sigma(\ef{F})\geq \frac{1}{2} \sigma(F)$.
\item[(a4)] $\avfs_Q\leq 2\avfs_{\pif{Q}}$.
\end{itemize}

Next we split the summation in the dual pairing by using the stopping cubes. Let $\pi(Q)=(F,G)$ denote that $\pif{Q}=F$ and $\pig{Q}=G$.

\begin{subequations}
\begin{align}
\pairo{T(f\sigma)}{g}&= \sum_{Q\in\cd} \lambda_Q \avs{f}_Q\av{\sigma}_Q\avo{g_Q}_Q\av{\omega}_Q\abs{Q}\nonumber\\\label{firstsummation}
&\leq\sum_{G\in\cg}(\sum_{\substack{F\in\cf\\F\subseteq G}}\sum_{\substack{Q\in\cd  \\\pi(Q)=(F,G)}} \lambda_Q\av{g_Q\omega}_Q\int_Qf\sigma )\\&\phantom{=}+\sum_{F\in\cf}(\sum_{\substack{G\in\cg\\G\subseteq F}}\sum_{\substack{Q\in\cd \\ \pi(Q)=(F,G)}} \lambda_Q\avs{f}_Q\av{\sigma}_Q\int_Q g_Q\omega).\label{secondsummation}
\end{align}
\end{subequations}
\begin{remark}
As explained in Remark \ref{operators}, in the case $r=1$ we may deal symmetrically with the pairs $(f,\sigma)$ and $(g,\omega)$. Hence in the case $r=1$ we may impose the stopping condition 
$$
\avo{g}_{G'}>2\avo{g}_{G}.
$$ for the real-valued function $g$, as it is done in Hyt\"onen's proof \cite[Section 6]{hytonen2012} of the case $r=1$, whereas in the case $1<r<\infty$ we reduce the sequence-valued function $g=(g_Q)$ to the piecewise constant function $g=(a_Q1_Q)$ and impose the stopping condition $$
\normrp{(a_Q)_{{\begin{subarray}{l}Q\in\cd :\\ Q\supseteq G'\end{subarray}}}}>2\avgg.
$$
\end{remark}
\subsection{Lemma} 
The following well-known lemma will be used in Section \ref{applyingthedualtestingcondition} and in Section \ref{applyingthedirecttestingcondition}.
\begin{lemma}[Special case of dyadic Carleson embedding theorem]\label{carlesonembedding}Let $1<p<\infty$. Suppose that $\{\ef{F}\}_{F\in\cf}$ is a pairwise disjoint collection such that for each $F\in\cf$ we have $\ef{F}\subseteq F$ and $\sigma(F)\leq 2 \sigma(\ef{F})$. Then
$$
(\sum_{F\in\cf}(\avs{\abs{f}}_{F})^p\sigma(F))^{1/p}\leq 2^{1/p}p'\norm{f}_{L^p(\sigma)}.
$$
\end{lemma}
\begin{proof}[Proof of the lemma]
By the definition of the Hardy-Littlewood maximal function we have $\avs{\abs{f}}_{F}\leq \inf_{F}M^\sigma f$. Moreover we have the norm inequality $\norm{M^\sigma f}_{L^p(\sigma)}\leq p' \norm{f}_{L^p(\sigma)}$. These facts together with the assumptions
yield $$
(\sum_{F\in\cf}(\avs{\abs{f}}_{F})^p\sigma(F))^{1/p}\leq 2^{1/p}(\sum_{F\in\cf}\int_{\ef{F}}(\inf_{F} M^\sigma f ) \sigma)^{1/p}\leq 2^{1/p} \norm{M^\sigma f}_{L^p(\sigma)}\leq 2^{1/p}p'\norm{f}_{L^p(\sigma)}.
$$
\end{proof}
\subsection{Applying the dual testing condition}\label{applyingthedualtestingcondition}
Let us first consider the summation \eqref{firstsummation}. Assume for the moment that we may replace $f$ in the summation \eqref{firstsummation} with functions $f_G$ that satisfy 
\begin{equation}\label{fgclaim}
(\sum_{G\in\cg}\norm{f_G}_{L^p(\sigma)}^p)^{1/p}\leq 5p'\norm{f}_{L^p(\sigma)}.
\end{equation}
Then we have
\begin{align*}
&\sum_{G\in\cg}\sum_{\substack{F\in\cf\\F\subseteq G}}\sum_{\substack{Q\in\cd  \\\pi(Q)=(F,G)}} \lambda_Q\av{g\omega}_Q\int_Qf_G\sigma &&\\
&\leq \sum_{G\in\cg}\int (\sum_{\substack{Q\in\cd:\\\pig{Q}=G}} \lambda_Q\av{g_Q\omega}_Q1_Q) f_G\sigma&&\text{by relaxing the summation condition}\\
&= \sum_{G\in\cg}\int (\sum_{Q\in\cd} \lambda_Q\av{(g_G)_Q\omega}_Q1_Q) f_G\sigma&&\text{by defining $g_G:=(g_Q)_{{\begin{subarray}{l}Q\in\cd:\\\pig{Q}=G\end{subarray}}}$}\\
&=\sum_{G\in\cg} \pairs{f_G}{T^*_G(g_G\;\omega)}&&\\
&\leq \sum_{G\in\cg} \norm{f_G}_{L^p(\sigma)}\norm{T^*_G(g_G\;\omega)}_{L^{p'}(\sigma)}&&\text{by H\"older's inequality}\\
&\leq C^* \sum_{G\in\cg} \norm{f_G}_{L^p(\sigma)} \norm{g_G}_{L^{\infty}_{\ell^{r'}}(\omega)}\,\omega(G)^{1/{p'}}&&\text{by the testing condition \eqref{dualtesting}}\\
&\leq 2 C^* \sum_{G\in\cg}\norm{f_G}_{L^p(\sigma)}\avon{g}_G\omega(G)^{1/{p'}}&&\text{by the property (b5)}\\
&\leq 2C^* (\sum_{G\in\cg}\norm{f_G}_{L^p(\sigma)}^p)^{1/p}(\sum_{G\in\cg}(\avon{g}_G)^{p'}\omega(G))^{1/{p'}}&&\text{by H\"older's inequality}\\
&\leq 2^{1+1/{p'}}pC^* (\sum_{G\in\cg}\norm{f_G}_{L^p(\sigma)}^p)^{1/p}\norm{g}_{L^{p'}_{\ell^{r'}}(\omega)}&&\text{by Lemma \ref{carlesonembedding}}\\
&\leq 4pC^* 5p' \norm{f}_{L^p(\sigma)}\norm{g}_{L^{p'}_{\ell^{r'}}(\omega)}&&\text{by the claimed inequality \eqref{fgclaim}}.
\end{align*}
Next we prove that we may replace $f$ in the summation \eqref{firstsummation} with functions $f_G$ that satisfy the claimed inequality \eqref{fgclaim}.
\begin{claim}
In the summation \eqref{firstsummation} we may replace $f$ with functions $f_G$ that satisfy 
\begin{equation*}
(\sum_{G\in\cg}\norm{f_G}_{L^p(\sigma)}^p)^{1/p}\leq 5p'\norm{f}_{L^p(\sigma)}.
\end{equation*}
\end{claim}
\begin{proof}[Proof of the claim]
Since the summation condition $\pig{Q}=G$ implies that $Q\subseteq G$ and since the sets $G'\in\chg{G}$ and $\eg{G}$ partition $G$, we have 
$$
\int_Qf\sigma=\int_{Q\cap\eg{G}}f\sigma+\sum_{G'\in\chg{G}}\int_{Q\cap G'}f\sigma.
$$
We may suppose that $Q\cap G'\neq\emptyset$ because otherwise the integral over $Q\cap G'$ vanishes. Then either $G'\subsetneq Q$ or $Q\subseteq G'$, the latter which is excluded by the summation condition $\pig{Q}=G$. Hence we may restrict the summation index set $\{G'\in\chg{G}\}$ to the set $\{G'\in\chg{G} :G'\subsetneq Q\}$. Therefore 
$$
\int_{Q\cap G'}f\sigma=\int_{G'}f\sigma=\avs{f}_{G'}\sigma(G')=\int_Q \avs{f}_{G'}1_{G'}\sigma.
$$
The summation conditions $\pif{Q}=F$ and $F\subseteq G$ imply that $Q \subseteq F\subseteq G$. Therefore \begin{equation}\label{def_chsg}\begin{split}\{G'\in\chg{G} : G'\subsetneq Q\}&\subseteq\{G'\in\chg{G} : G'\subseteq F\subseteq G \text{ for some } F\in\cf\}\\
&=\bigcup_{\substack{F\in\cf:\\\pig{F}=G}} \{G'\in\chg{G} : \pif{G'}=F\}=:\chsg{G}.\end{split}\end{equation}
Altogether we have
$$
\int_Qf\sigma\leq \int_Q(f1_{\eg{G}}+\sum_{G'\in\chsg{G}} \avs{f}_{G'}1_{G'})\sigma=:\int_Q f_G\sigma.
$$
Next we check the claimed inequality \eqref{fgclaim}. 
By the triangle inequality we have
$$
\norm{f_G}_{L^p(\sigma)}\leq \norm{f1_{\eg{G}}}_{L^p(\sigma)}+\norm{\sum_{G'\in\chsg{G}} \avs{f}_{G'}1_{G'}}_{L^p(\sigma)},
$$ which by the triangle inequality and by the pairwise disjointness of each of the collections $\{G'\}_{G'\in\chsg{G}}$ and $\{\eg{G}\}_{G\in\cg}$ implies that
\begin{equation*}
\begin{split}
(\sum_{G\in\cg}\norm{f_G}_{L^p(\sigma)}^p)^{1/p}&\leq (\sum_{G\in\cg}\norm{f1_{\eg{G}}}_{L^p(\sigma)}^p)^{1/p}+
(\sum_{G\in\cg}\norm{\sum_{G'\in\chsg{G}} \avs{f}_{G'}1_{G'}}_{L^p(\sigma)}^p)^{1/p}\\
&\leq (\norm{\sum_{G\in\cg}f1_{\eg{G}}}_{L^p(\sigma)}^p)^{1/p}+
(\sum_{G\in\cg}\sum_{G'\in\chsg{G}}\norm{ \avs{f}_{G'}1_{G'}}_{L^p(\sigma)}^p)^{1/p}\\
&\leq \norm{f}_{L^p(\sigma)}+
(\sum_{G\in\cg}\sum_{G'\in\chsg{G}}(\avs{f}_{G'})^p\sigma(G'))^{1/p}.
\end{split}
\end{equation*}
We can estimate the last term as follows.
\begin{align*}
&\sum_{G\in\cg}\sum_{G'\in\chsg{G}} (\avs{f}_{G'})^{p}\sigma(G') &&\\
&= \sum_{G\in\cg}\sum_{\substack{F\in\cf:\\\pig{F}=G}}\sum_{\substack{G'\in\chg{G}:\\\pif{G'}=F}} (\avs{f}_{G'})^{p}\sigma(G')&&\text{by \eqref{def_chsg}}\\
&\leq \sum_{G\in\cg}\sum_{\substack{F\in\cf:\\\pig{F}=G}}2^p (\avs{f}_{F})^{p}(\sum_{\substack{G'\in\chg{G}:\\\pif{G'}=F}} \sigma(G'))&&\text{by the property (a4)}\\
&\leq 2^p \sum_{G\in\cg}\sum_{\substack{F\in\cf:\\\pig{F}=G}}(\avs{f}_{F})^{p}\sigma(F)&&\text{because $\chg{G}$ is pairwise disjoint}\\
&= 2^p\sum_{F\in\cf} (\avs{f}_{F})^{p}\sigma(F)&&\text{because $\cf=\bigcup_{G\in\cg}\{F\in\cf : \pig{F}=G\}$ }\\
&\leq 2^{p+1} (p')^p\norm{f}_{L^p(\sigma)}^{p}&&\text{by Lemma \ref{carlesonembedding}}.
\end{align*}
Altogether
$$
(\sum_{G\in\cg}\norm{f_G}_{L^p(\sigma)}^p)^{1/p}\leq \norm{f}_{L^p(\sigma)}+2^{1/p+1}p'\norm{f}_{L^p(\sigma)}\leq 5p'\norm{f}_{L^p(\sigma)}.
$$
This concludes the proof of the claim.
\end{proof}
\subsection{Applying the direct testing condition}\label{applyingthedirecttestingcondition}

Next we estimate the summation \eqref{secondsummation}. 
\begin{align*}
&\sum_{F\in\cf}\sum_{\substack{G\in\cg\\G\subseteq F}}\sum_{\substack{Q\in\cd \\ \pi(Q)=(F,G)}} \lambda_Q\avs{f}_Q\av{\sigma}_Q\int_Q g_Q\omega &&\\
&\leq 2 \sum_{F\in\cf}\avs{f}_F\sum_{\substack{G\in\cg\\G\subseteq F}}\sum_{\substack{Q\in\cd \\ \pi(Q)=(F,G)}} \lambda_Q\av{\sigma}_Q\int_Q g_Q\omega&& \text{by the property (a4)}\\
&= 2 \sum_{F\in\cf}\avs{f}_F \int \sum_{Q\in\cd}(T_F(\sigma))_Q (g_F)_{Q}\omega &&\text{by $g_F:=(g_Q)_{{\begin{subarray}{l}Q\in\cd :\\ \pi(Q)=(F,G) \text{ for some } \\G\in\cg \text{ such that } G\subseteq F\end{subarray}}}$}\\
&=2\sum_{F\in\cf}\avs{f}_F\pairo{T_F({\sigma})}{g_F}&&\\
&\leq 2 \sum_{F\in\cf}\avs{f}_F\norm{T_F({\sigma})}_{L^p_{\ell^r}(\omega)} \norm{g_F}_{L^{p'}_{\ell^{r'}}(\omega)}&&\text{by H\"older's inequality}\\
&\leq 2 C\sum_{F\in\cf}\avs{f}_F\sigma(F)^{1/p}\norm{g_F}_{L^{p'}_{\ell^{r'}}(\omega)}&&\text{by the testing condition \eqref{testing}}\\
&\leq 2 C(\sum_{F\in\cf}(\avs{f}_F)^p\sigma(F))^{1/p} (\sum_{F\in\cf}\norm{g_F}^{p'}_{L^{p'}_{\ell^{r'}}(\omega)})^{1/{p'}}&&\text{by H\"older's inequality}\\
&\leq 2 C 2 p' \norm{f}_{L^p(\sigma)}(\sum_{F\in\cf}\norm{g_F}^{p'}_{L^{p'}_{\ell^{r'}}(\omega)})^{1/{p'}}&&\text{by Lemma \ref{carlesonembedding}}.
\end{align*}
The proof of the following claim completes the proof of the theorem.
\begin{claim}
We have 
\begin{equation}\label{claimforgf}
(\sum_{F\in\cf}\norm{g_F}^{p'}_{L^{p'}_{\ell^{r'}}(\omega)})^{1/{p'}}\leq 5p\norm{g}_{L^{p'}_{\ell^{r'}}(\omega)}.
\end{equation}
\end{claim}
\begin{proof}[Proof of the claim]
By definition the components of the function $g_F=(a_Q1_Q)_{Q\in\ci(F)}$ are indexed by the set
$$
\ci(F)=\{Q\in\cd : \pi(Q)=(F,G) \text{ for some } G\in\cg \text{ such that } G\subseteq F \}.
$$
The function $g_F$ is supported on $\bigcup_{Q\in\ci(F)}Q$. Since the condition $\pif{Q}=F$ implies that $Q\subseteq F$, we have that $\bigcup_{Q\in\ci(F)}Q\subseteq F$. Since the sets $F'\in\chf{F}$ and $\ef{F}$ partition $F$, we have 
$$
g_F= g_F1_{\ef{F}} +\sum_{F'\in\chf{F}}g_F1_{F'}.
$$
By the triangle inequality we have
$$
\norm{g_F}_{L^{p'}_{\ell^{r'}}(\omega)}\leq \norm{g_F1_{\ef{F}}}_{L^{p'}_{\ell^{r'}}(\omega)}+\norm{\sum_{F'\in\chf{F}}g_F1_{F'}}_{L^{p'}_{\ell^{r'}}(\omega)},
$$
which by the triangle inequality, by the fact $\normrp{g_F}\leq\normrp{g}$ and by the pairwise disjointness of each of the collections  $\{\ef{F}\}_{F\in\cf}$ and $\{F'\}_{F'\in\chf{F}}$  implies that
\begin{equation*}
\begin{split}
(\sum_{F\in\cf}\norm{g_F}^{p'}_{L^{p'}_{\ell^{r'}}(\omega)})^{1/{p'}}&\leq (\sum_{F\in\cf}\norm{g_F1_{\ef{F}}}^{p'}_{L^{p'}_{\ell^{r'}}(\omega)})^{1/{p'}}+(\sum_{F\in\cf}\norm{\sum_{F'\in\chf{F}}g_F1_{F'}}^{p'}_{L^{p'}_{\ell^{r'}}(\omega)})^{1/{p'}}\\
&\leq (\norm{\sum_{F\in\cf}g1_{\ef{F}}}^{p'}_{L^{p'}_{\ell^{r'}}(\omega)})^{1/{p'}}+(\sum_{F\in\cf}\sum_{F'\in\chf{F}}\norm{g_F1_{F'}}^{p'}_{L^{p'}_{\ell^{r'}}(\omega)})^{1/{p'}}\\
&\leq \norm{g}_{L^{p'}_{\ell^{r'}}(\omega)}+(\sum_{F\in\cf}\sum_{F'\in\chf{F}}\norm{g_F1_{F'}}^{p'}_{L^{p'}_{\ell^{r'}}(\omega)})^{1/{p'}}.
\end{split}
\end{equation*}
It remains to estimate the last term. Consider the integral
$$
\norm{g_F1_{F'}}^{p'}_{L^{p'}_{\ell^{r'}}(\omega)}=\int_{F'}\normrp{g_F}^{p'}=\begin{cases}\int_{F'} ((\sum_{Q\in\ci(F)} a_Q^{r'}1_Q(x))^{1/r'})^{p'}\omega(x) \mathrm{d}x& \text{ if } 1\leq r'<\infty,\\
\int_{F'}(\sup_{Q\in\ci(F)}(a_Q1_Q(x)))^{p'}\omega(x) \mathrm{d}x&\text{ if } r'=\infty.
\end{cases}
$$
Let $Q\in\ci(F)$ and $F'\in\chf{F}$. The cubes $Q$ and $F'$ for which $Q\cap F'=\emptyset$ do not contribute to the integral. Hence we may restrict to the cubes such that $Q\cap F'\neq\emptyset$. Then by nestedness either $F'\subsetneq Q$ or $Q\subseteq F'$, the latter which is excluded by the condition $\pif{Q}=F$. Hence $F'\subsetneq Q$. Moreover, we have that $\pig{Q}=G$ for some $G\subseteq F$, which implies that $Q\subseteq G\subseteq F$. Altogether we have $F'\subsetneq Q\subseteq G\subseteq F$. Therefore we may replace the summation over the index set $\chf{F}$ with the summation over the set
\begin{equation}
\begin{split}
\label{replace1}
\chsf{F}&=\{F'\in\chf{F} : F'\subseteq G\subseteq F \text{ for some } G\in\cg\}\\
&=\bigcup_{\substack{G\in\cg:\\\pif{G}=F}} \{F'\in\chf{F} : \pig{F'}=G\}.
\end{split}
\end{equation}
and we may replace the index set $\ci(F)$ with the index set
\begin{equation}
\label{replace2}
\ci(F,F'):=\{Q\in\cd : Q\supsetneq F' \text{ and } \pi(Q)=(F,G) \text{ for some } G\in\cg \text{ such that } G\subseteq F \}.
\end{equation}
By the containment $$\ci(F,F')\subseteq \{Q\in\cd : Q\supsetneq F'\}$$ 
and the property (b4) we have 
\begin{equation}
\label{estimateab}
\normrp{(a_Q)_{Q\in\ci(F,F')}}\leq \normrp{(a_Q)_{\begin{subarray}{l}Q\in\cd:\\Q\supsetneq F'\end{subarray}}}\leq 2 \avo{\normrp{g}}_{\pig{F'}}.
\end{equation} 
Therefore 
\begin{align*}
&\sum_{F\in\cf}\sum_{F'\in\chf{F}}\norm{g_F1_{F'}}^{p'}_{L^{p'}_{\ell^{r'}}(\omega)} &&\\
&=\sum_{F\in\cf}\sum_{F'\in\chsf{F}}\norm{(a_Q)_{Q\in\ci(F,F')}1_{F'}}^{p'}_{L^{p'}_{\ell^{r'}}(\omega)} &&\text{the replacements \eqref{replace1} and \eqref{replace2}}\\
&\leq \sum_{F\in\cf}
\sum_{F'\in\chsf{F}}2^{p'} (\avo{\normrp{g}}_{\pig{F'}})^{p'}\omega(F')&&\text{by \eqref{estimateab}}\\
&\leq \sum_{F\in\cf}
\sum_{\substack{G\in\cg\\\pif{G}=F}} \sum_{\substack{F'\in\chf{F}\\\pig{F'}=G}}2^{p'} (\avo{\normrp{g}}_{\pig{F'}})^{p'}\omega(F')&&\text{by \eqref{replace1}}\\
&= 2^{p'}\sum_{F\in\cf}
\sum_{\substack{G\in\cg\\\pif{G}=F}}  (\avo{\normrp{g}}_{G})^{p'}(\sum_{\substack{F'\in\chf{F}\\\pig{F'}=G}}\omega(F'))\\
&\leq {2^{p'}}\sum_{F\in\cf}
\sum_{\substack{G\in\cg\\\pif{G}=F}}  (\avo{\normrp{g}}_{G})^{p'}\omega(G)&&\text{because $\chf{F}$ is pairwise disjoint}\\
&= {2^{p'}}\sum_{G\in\cg}(\avo{\normrp{g}}_{G})^{p'}\omega(G)&&\text{because $\cg=\bigcup_{F\in\cf}\{G\in\cg : \pif{G}=F\}$}\\
&\leq 2^{p'+1}p^{p'}\norm{g}^{p'}_{L^{p'}_{\ell^{r'}}(\omega)}&&\text{by Lemma \ref{carlesonembedding}}.
\end{align*}
Altogether 
$$
(\sum_{F\in\cf}\norm{g_F}^{p'}_{L^{p'}_{\ell^{r'}}(\omega)})^{1/{p'}}\leq \norm{g}_{L^{p'}_{\ell^{r'}}(\omega)}+2^{1/{p'}+1}p \norm{g}_{L^{p'}_{\ell^{r'}}(\omega)}\leq 5p\norm{g}_{L^{p'}_{\ell^{r'}}(\omega)}.
$$
This completes the proof of the claim.
\end{proof}

\end{proof}
\begin{remark} In fact each of the proofs \cite{lacey2009,treil2012,hytonen2012} for $r=1$, the proof \cite{scurry2010} for $1<r<\infty$, and the proof \cite{sawyer1982} for $r=\infty$ each works in the case $T(\,\cdot\,\omega):L^p(\sigma)\to L^q_{\ell^r}(\omega)$ with $1<p\leq q<\infty$. Also the proof of this note works in that case by using the following facts. For $p'\geq q'$ the estimate $\norm{\,\cdot\,}_{\ell^{p'}}\leq \norm{\,\cdot\,}_{\ell^{q'}}$  implies that
$$
(\sum_{G\in\cg}(\avon{g}_G)^{p'}\omega(G))^{1/{p'}}\leq (\sum_{G\in\cg}(\avon{g}_G)^{q'}\omega(G))^{1/{q'}}$$
and that 
 $$
(\sum_{F\in\cf}\norm{g_F}^{p'}_{L^{q'}_{\ell^{r'}}(\omega)})^{1/{p'}}\leq (\sum_{F\in\cf}\norm{g_F}^{q'}_{L^{q'}_{\ell^{r'}}(\omega)})^{1/{q'}}.
$$
Moreover, the estimate \eqref{claimforgf} holds for every $p'$, hence in particular for $q'$, as it is seen from the proof of the estimate. 
\end{remark}
\subsection*{Acknowledgment}
The author would like to thank Tuomas Hyt\"onen for suggesting the author to study whether Hyt\"onen's proof \cite[Section 6]{hytonen2012} of the case $r=1$ could be adapted to the case $1<r<\infty$, for pointing out how to simplify certain details of the proof, and for  encouraging and insightful discussions.

\bibliographystyle{plain}
\bibliography{references_two_weights}
\end{document}